\documentclass{amsart}
\usepackage{rorabaugh}
\usepackage{amsaddr}
\usepackage[margin=1.25in]{geometry}
\title{Density Dichotomy in Random Words}
\author[J. Cooper and D. Rorabaugh]
       {Joshua Cooper \& Danny Rorabaugh}
\address{Department of Mathematics, University of South Carolina \\ 
Department of Mathematics and Statistics, Queen's University}

\allowdisplaybreaks

\begin{document}

\setcounter{page}{1}
\thispagestyle{empty}


\label{firstpage}

\begin{abstract}
Word $W$ is said to \emph{encounter} word $V$ provided there is a homomorphism $\phi$ mapping letters to nonempty words so that $\phi(V)$ is a substring of $W$. For example, taking $\phi$ such that $\phi(h)=c$ and $\phi(u)=ien$, we see that ``science'' encounters ``huh'' since $cienc=\phi(huh)$. The density of $V$ in $W$, $\delta(V,W)$, is the proportion of substrings of $W$ that are homomorphic images of $V$. So the density of ``huh'' in ``science'' is $2/{8 \choose 2}$. A word is \emph{doubled} if every letter that appears in the word appears at least twice.

The dichotomy: Let $V$ be a word over any alphabet, $\Sigma$ a finite alphabet with at least 2 letters, and $W_n \in \Sigma^n$ chosen uniformly at random. Word $V$ is doubled if and only if $\E(\delta(V,W_n)) \rightarrow 0$ as $n \rightarrow \infty$.

We further explore convergence for nondoubled words and concentration of the limit distribution for doubled words around its mean.
\end{abstract}

\maketitle

\section{Introduction}

Graph densities provide the basis for many recent advances in extremal graph theory and the limit theory of graph (see Lov\'asz \cite{L-12}). 
To see if this paradigm is similarly productive for other discrete structures, we here explore pattern densities in free words. 
In particular, we consider the asymptotic densities of a fixed pattern in random words as a first step in developing the combinatorial limit theory of free words. 

\subsection{Definitions}

Free words (or simply, words) are elements of the semigroup formed from a nonempty alphabet $\Sigma$ with the binary operation of concatenation, denoted by juxtaposition, and with the empty word $\varepsilon$ as the identity element. 
The set of all finite words over $\Sigma$ is $\Sigma^*$ and the set of $\Sigma$-words of length $k \in \N$ is $\Sigma^k$. 
For alphabets $\Gamma$ and $\Sigma$, a homomorphism $\phi: \Gamma^* \rightarrow \Sigma^*$ is uniquely defined by a function $\phi:\Gamma \rightarrow \Sigma^*$. 
We call a homomorphism \emph{nonerasing} provided it is defined by $\phi:\Gamma \rightarrow \Sigma^* \setminus \{\varepsilon\}$; that is, no letter maps to $\varepsilon$, the empty word.

Let $V$ and $W$ be words.
The \emph{length} of $W$, denoted $|W|$, is the number of letters in $W$, including multiplicity.
Denote with ${\rm L}(W)$ the set of letters found in $W$ and with $||W||$ the number of letter repeats in $W$, so $|W| = |{\rm L}(W)| + ||W||$.
For example $|banana| = 6$, ${\rm L}(banana) = \{a,b,n\}$, and $||banana|| = 3$.
$W$ has ${|W|+1 \choose 2}$ \emph{substrings}, each defined by an ordered pair $(i,j)$ with $0\leq i < j \leq |W|$.
Denote with $W[i,j]$ the word found in the $(i,j)$-substring, which consists of $j-i$ consecutive letters of $W$, beginning with the $(i+1)$-th.
$V$ is a \emph{factor} of $W$, denoted $V \leq W$, provided $V = W[i,j]$ for some $0\leq i < j \leq |W|$; that is, $W = SVT$ for some (possibly empty) words $S$ and $T$.
For example, $banana[2,6] = nana \leq banana$.

$W$ is an \emph{instance} of $V$, or \emph{$V$-instance}, provided there exists a nonerasing homomorphism $\phi$ such that $W = \phi(V)$. 
(Here $V$ is sometimes referred to as a \emph{pattern} or \emph{pattern word}).
For example, $banana$ is an instance of $cool$ using homomorphism $\phi$ defined by $\phi(c)=b$, $\phi(o)=an$, and $\phi(l)=a$.
$W$ \emph{encounters} $V$, denoted $V \preceq W$, provided $W'$ is an instance of $V$ for some factor $W' \leq W$.
For example $cool \preceq bananasplit$.
For $W \neq \varepsilon$, denote with $\delta(V,W)$ the proportion of substrings of $W$ that give instances of $V$.
For example, $\delta(xx, banana) = 2/{7 \choose 2}$. 
$\delta_{sur}(V,W)$ is the characteristic function for the event that $W$ is an instance of $V$.

Fix alphabets $\Gamma$ and $\Sigma$.
An \emph{encounter} of $V$ in $W$ is an ordered triple $(a,b,\phi)$ where $W[a,b] = \phi(V)$ for homomorphism $\phi:\Gamma^* \rightarrow \Sigma^*$.
When $\Gamma = {\rm L}(V)$ and $W \in  \Sigma^*$, denote with $\hom(V,W)$ the number of encounters of $V$ in $W$.
For example, $\hom(ab,cde) = 4$ since $cde[0,2]$ and $cde[1,3]$ are instances of $ab$, each for one homomorphism $\{a,b\}^* \rightarrow \{c,d,e\}^*$, and $cde[0,3]$ is an instance of $ab$ under two homomorphisms.
Note that the conditions on $\Gamma$ and $\Sigma$ are necessary for $\hom(V,W)$ to not be 0 or $\infty$.

\begin{fact}
	For fixed words $V$ and $W \neq \varepsilon$, 
	\[ {|W| + 1 \choose 2}\delta(V,W) \leq \hom(V,W). \]
\end{fact}

\subsection{Background}

Word encounters have primarily been explored from the perspective of avoidance.
Word $W$ \emph{avoids} a (pattern) word $V$ provided $V \not \preceq W$.
$V$ is \emph{$k$-avoidable} provided, from a $k$-letter alphabet, there are infinitely many words that avoid $V$.
The premier result on word avoidance is generally considered to be the proof of Thue \cite{T-06} that the word $aa$ is 3-avoidable but not 2-avoidable.
Two seminal papers on avoidability, by Bean, Ehrenfeucht, and McNulty \cite{BEM-79} and Zimin \cite{Z-82,Z-84}, include classification of unavoidable words--that is, words that are not $k$-avoidable for any $k$.
Recently, the authors \cite{CR-14} and Tao \cite{T-14} investigated bounds on the length of words that avoid unavoidable words.
There remain a number of open problems regarding which words are $k$-avoidable for particular $k$.
See Lothaire \cite{L-02} and Currie \cite{C-05} for surveys on avoidability results and Blanchet-Sadri and Woodhouse \cite{BW-12} for recent work on 3-avoidability.

A word is \emph{doubled} provided every letter in the word occurs at least twice. 
Otherwise, if there is a letter that occurs exactly once, we say the word is \emph{nondoubled}
Every doubled word is $k$-avoidable for some $k>1$ \cite{L-02}.
For a doubled word $V$ with $k \geq 2$ distinct letters and an alphabet $\Sigma$ with $|\Sigma| = q \geq 4$, $(k,q) \neq (2,4)$, Bell and Goh \cite{BG-07} showed that there are at least $\lambda(k,q)^n$ words in $\Sigma^n$ that avoid $V$, where 
\[ \lambda(k,q) = m\left(1 + \frac{1}{(m-2)^k}\right)^{-1}. \]
This exponential lower bound on the number of words avoiding a doubled word hints at the moral of the present work: instances of doubled words are rare.
For a doubled word $V$ and an alphabet $\Sigma$ with at least 2 letters, the probability that a random word $W_n \in \Sigma^n$ avoids $V$ is asymptotically 0. 
Indeed, the event that $W_n[b|V|,(b+1)|V|]$ is an instance of $V$ has nonzero probability and is independent for distinct $b$.
Nevertheless, $\delta(V,W_n)$, the proportion of substrings of $W$ that are instances of $V$, is asymptotically negligible.

\section{The Dichotomy}

In this section, we establish a density-motivated bipartition of all free words into doubled and nondoubled words. 
From there, we present a more detailed analysis of the asymptotic densities in these two classes. 

\begin{theorem} \label{main}
	Let $V$ be a word on any alphabet. Fix an alphabet $\Sigma$ with $q\geq2$ letters, and let $W_n \in \Sigma^n$ be chosen uniformly at random. The following are equivalent:
	\begin{enumerate}
		\renewcommand{\theenumi}{\roman{enumi}}
		\item $V$ is doubled (that is, every letter in $V$ occurs at least twice);
		\item $\lim_{n\rightarrow \infty} \E(\delta(V,W_n)) = 0$.
	\end{enumerate}
\end{theorem}

\begin{proof}
	First we prove $(i) \Longrightarrow (ii)$. 
	Note that in $W_n$, there are in expectation the same number of encounters of $V$ as there are of any anagram of $V$. 
	Indeed, if $V'$ is an anagram of $V$ and $\phi$ is a nonerasing homomorphism, then $|\phi(V')| = |\phi(V)|$.
	\begin{fact} \label{anagram}
		If $V'$ is an anagram of $V$, then $\E(\hom(V,W_n))=\E(\hom(V',W_n))$.
	\end{fact}
	
	Assume $V$ is doubled and let $\Gamma = {\rm L}(V)$ and $k = |\Gamma|$.
	Given Fact \ref{anagram}, we consider an anagram $V' = XY$ of $V$, where $|X| = k$ and $\Gamma = {\rm L}(X) = {\rm L}(Y)$.
	That is, $X$ comprises one copy of each letter in $\Gamma$ and all the duplicate letters of $V$ are in $Y$.
	
	We obtain an upper bound for the average density of $V$ by estimating $\E(\hom(V',W_n))$. 
	To do so, sum over starting position $i$ and length $j$ of encounters of $X$ in $W_n$ that might extend to an encounter of $V'$.
	There are ${j+1 \choose k+1}$ homomorphisms $\phi$ that map $X$ to $W_n[i,i+j]$ and the probability that $W_n[i+j,i+j+|\phi(Y)|] = \phi(Y)$ is at most $q^{-j}$.
	Also, the series $\sum_{j = k}^{\infty} {j+1 \choose k+1} q^{-j}$ converges (try the ratio test) to some $c$ not dependent on $n$.
	\begin{eqnarray*}
		\E(\delta(V,W_n)) &\leq& \frac{1}{{n+1 \choose 2}}\E\left(\hom(V',W_n)\right) \\
		&<& \frac{1}{{n+1 \choose 2}} \sum_{i = 0}^{n - |V|} \sum_{j = k}^{n - i} {j+1 \choose k+1} q^{-j} \\
		&<& \frac{1}{{n+1 \choose 2}} \sum_{i = 0}^{n - |V|} c \\
		&=& \frac{c(n-|V|+1)}{{n+1 \choose 2}} \\
		&=& O(n^{-1}).
	\end{eqnarray*}
	
	We prove $(ii) \Longleftarrow (i)$ by contraposition. Assume there is a letter $x$ that occurs exactly once in $V$.
	Write $V = TxU$ where ${\rm L}(V) \setminus {\rm L}(TU) = \{x\}$.
	We obtain a lower bound for $\E(\delta(V,W_n))$ by only counting encounters with $|\phi(TU)| = |TU|$.
	Note that each such encounter is unique to its instance, preventing double-counting.
	For this undercount, we sum over encounters with $W_n[i,i+j]=\phi(x)$.
	\begin{eqnarray*}
		\E(\delta(V,W_n)) &=& \E(\delta(TxU,W_n)) \\
		&\geq&\frac{1}{{n+1 \choose 2}} \sum_{i = |T|}^{n-|U|-1} \sum_{j = 1}^{i-|T|} q^{-||TU||} \\
		&=&q^{-||TU||}\frac{1}{{n+1 \choose 2}} \sum_{i = |T|}^{n-|U|-1} (i - |T|) \\
		&=&q^{-||TU||}\frac{{n-|UT| \choose 2}}{{n+1 \choose 2}}  \\
		&\sim&q^{-||TU||}\\
		&>&0.
	\end{eqnarray*}

\end{proof}

It behooves us now to develop more precise theory for these two classes of words: doubled and nondoubled.
Lemma \ref{base} below both helps develop that theory and gives insight into the detrimental effect that letter repetition has on encounter frequency.

\begin{fact} \label{CRT}
	For $\overline{r} = \{r_1, \ldots, r_k\} \in (\Z^+)^k$ and $d = \gcd_{i \in [k]}(r_i)$, there exists integer $N = N_{\overline{r}}$ such that for every $n>N$ there exist coefficients $a_1, \ldots, a_k \in \Z^+$ such that $dn = \sum_{i=1}^{k} a_ir_i$ and $a_i \leq N$ for $i \geq 2$.
\end{fact}
%
%

\begin{lem} \label{base}
	For any word $V$, let $\Gamma = {\rm L}(V) = \{x_1, \ldots, x_k\}$ where $x_i$ has multiplicity $r_i$ for each $i \in [k]$. 
	Let $U$ be $V$ with all letters of multiplicity $r= \min_{i \in [k]}(r_i)$ removed. 
	Finally, let $\Sigma$ be any finite alphabet with $|\Sigma| = q\geq 2$ letters. 
	Then for a uniformly randomly chosen $V$-instance $W \in \Sigma^{dn}$, where $d = \gcd_{i \in [k]}(r_i)$, there is asymptotically almost surely a homomorphism $\phi: \Gamma^* \rightarrow \Sigma^*$ with $\phi(V) = W$ and $|\phi(U)| < \sqrt{dn}$.
\end{lem}

\begin{proof}
	Let $a_n$ be the number of $V$-instances in $\Sigma^n$ and $b_n$ be the number of homomorphisms $\phi : \Gamma^* \rightarrow \Sigma^*$ such that $|\phi(V)| =  n$.
	Let $b_n^1$ be the number of these $\phi$ such that $\phi(U) < \sqrt{n}$ and $b_n^2$ the number of all other $\phi$ so that $b_n = b_n^1 + b_n^2$.
	Similarly, let $a_n^1$ be the number of $V$-instances in $\Sigma^n$ for which there exists a $\phi$ counted by $b_n^1$ and $a_n^2$ the number of instances with no such $\phi$, so $a_n = a_n^1 + a_n^2$.
	Observe that $a_n^2 \leq b_n^2$.

	Without loss of generality, assume $r_1 = r$ (rearrange the $x_i$ if not). 
	We now utilize $N = N_{\overline{r}}$ from Proposition \ref{CRT}. 
	For sufficiently large $n$, we can undercount $a_{dn}^1$ by counting homomorphisms $\phi$ with $|\phi(x_i)| = a_i$ for the $a_i$ attained from Fact \ref{CRT}. 
	Indeed, distinct homomorphisms with the same image-length for every letter in $V$ produce distinct $V$-instances.
	Hence 

	\begin{eqnarray*}
		a_{dn}^1 & \geq & q^{\sum_{i=1}^k a_i} \\
		& \geq & q^{\left(\frac{dn - (k-1)N}{r} + r(k-1)\right)} \\
		& = & cq^{\left(\frac{dn}{r}\right)},
	\end{eqnarray*}
	where $c = q^{(k-1)(r^2-N)/r}$ depends on $V$ but not on $n$.	
	To overcount $b_n^2$ (and $a_{dn}^2$ by extension), we consider all ${n+1 \choose |V|+1}$ ways to partition an $n$-letter length and so determine the lengths of the images of the letters in $V$. 
	However, for letters with multiplicity strictly greater than $r$, the sum of the lengths of their images must be at least $\sqrt{n}$.
	
	\begin{eqnarray*}
		b_n^2 & \leq & {n+1 \choose |V|+1}  \sum_{i = \ceil{\sqrt{n}}}^n q^{\left(\frac{n - i}{r} + \frac{i}{r+1}\right)}\\
		& =  & {n+1 \choose |V|+1} \sum_{i = \ceil{\sqrt{n}}}^n q^{\left(\frac{n}{r} - \frac{i}{r(r+1)}\right)}\\
		& < & n^{|V|+2}  q^{\left(\frac{n}{r} - \frac{\sqrt{n}}{r(r+1)}\right)}\\
		& = & q^{\frac{n}{r}}o(1).\\
		\\
		a_{dn}^2 & \leq & b_{dn}^2\\
		& = & o(a_{dn}^1).
	\end{eqnarray*}
	
	That is, the proportion of $V$-instances of length $dn$ that cannot be expressed with $|\phi(U)| < \sqrt{dn}$ diminishes to 0 as $n$ grows.
\end{proof}
 
\section{Density of Nondoubled Words}

In Theorem \ref{main}, we showed that the density of nondoubled $V$ in long random words (over a fixed alphabet with at least two letters) does not approach 0. 
The natural follow-up question is: Does the density converge?
To answer this question, we first prove the following lemma. 
Fixing $V = TxU$ where $x$ is a nonrecurring letter in $V$, the lemma tells us that all but a diminishing proportion of $V$-instances can be obtained by some $\phi$ with $|\phi(TU)|$ negligible.

\begin{lem} \label{bulk}
	Let $V = U_0x_1U_1x_2\cdots x_rU_r$ with $r\geq 1$, where $U = U_0U_1\cdots U_r$ is doubled with $k$ distinct letters (though any particular $U_j$ may be the empty word), the $x_i$ are distinct, and no $x_i$ occurs in $U$. 
	Further, let $\Gamma$ be the $(k+r)$-letter alphabet of $V$ and let $\Sigma$ be any finite alphabet with $q\geq 2$ letters. 
	Then there exists a nondecreasing function $g(n) = o(n)$ such that, for a randomly chosen $V$-instance $W \in \Sigma^n$, there is asymptotically almost surely a homomorphism $\phi: \Gamma^* \rightarrow \Sigma^*$ with $\phi(V) = W$ and $|\phi(x_r)| > n - g(n)$.
\end{lem}

\begin{proof}
	Let $X_i = x_1x_2\cdots x_i$ for $0 \leq i \leq r$ (so $X_0 = \varepsilon$). 
	For any word $W$, let $\Phi_W$ be the set of homomorphisms $\{\phi:  \Gamma^* \rightarrow \Sigma^* \mid \phi(V)=W\}$ that map $V$ onto $W$. 
	Define ${\bf P}_i$ to be the following proposition for $i \in [r]$:
	\begin{quotation}
		There exists a nondecreasing function $f_i(n) = o(n)$ such that, for a randomly chosen $V$-instance $W \in \Sigma^n$, there is asymptotically almost surely a homomorphism $\phi \in \Phi_W$ such that $|\phi(UX_{i-1})| \leq f_i(n)$.
	\end{quotation}
	
	The conclusion of this lemma is an immediate consequence of ${\bf P}_r$, with $g(n) = f_r(n)$, which we will prove by induction.
	Lemma \ref{base} provides the base case, with $r = 1$ and $f_1(n) = \sqrt{n}$.
	
	Let us prove the inductive step: ${\bf P}_i$ implies ${\bf P}_{i+1}$ for $i \in [r-1]$.	
	Roughly speaking, this says: If most instances of $V$ can be made with a homomorphism $\phi$ where $|\phi(UX_{i-1})|$ is negligible, then most instances of $V$ can be made with a homomorphism $\phi$ where $|\phi(UX_{i})|$ is negligible.
	
	Assume ${\bf P}_{i}$ for some $i \in [r-1]$, and set $f(n) = f_{i}(n)$. 
	Let $A_n$ be the set of $V$-instances in $\Sigma^n$ such that $|\phi(UX_{i-1})| \leq f(n)$ for some $\phi \in \Phi_W$. 
	Let $B_n$ be the set of all other $V$-instances in $\Sigma^n$. 
${\bf P}_{i}$ implies $|B_n| = o(|A_n|)$.
	
	Case 1: $U_{i} = \varepsilon$, so $x_{i}$ and $x_{i+1}$ are consecutive in $V$. 
	When $|\phi(UX_{i-1})| \leq f(n)$, we can define $\psi$ so that $\psi(x_{i}x_{i+1}) = \phi(x_{i}x_{i+1})$ and $|\psi(x_{i})|=1$; otherwise, let $\psi(y) = \phi(y)$ for $y \in\Gamma \setminus \{x_i,x_{i+1}\}$.
	Then $|\phi(UX_i)| \leq f(n)+1$ and ${\bf P}_{i+1}$ with $f_{i+1}(n) = f_i(n) + 1$.

	Case 2: $U_i \neq \varepsilon$, so $|U_i| > 0$.
	Let $g(n)$ be some nondecreasing function such that $f(n) = o(g(n))$ and $g(n) = o(n)$. 
(This will be the $f_{i+1}$ for ${\bf P}_{i+1}$.)
	Let $A_n^\alpha$ consist of $W \in A_n$ such that $|\phi(UX_{i})| \leq g(n)$ for some $\phi\in \Phi_W$. Let $A_n^\beta = A_n \setminus A_n^\alpha$. 
The objective henceforth is to show that $|A_n^\beta| = o(|A_n^\alpha|)$.
	
	For $Y \in A_n^\beta$, let $\Phi_Y^\beta$ be the set of homomorphisms $\{\phi \in \Phi_Y : |\phi(UX_{i-1})| \leq f(n)\}$ that disqualify $Y$ from being in $B_n$. 
	Hence $Y \in A_n$ implies $\Phi_Y^\beta \neq \emptyset$.
	Since $Y \not \in A_n^\alpha$, $\phi \in \Phi_Y^\beta$ implies $|\phi(UX_{i})| > g(n)$, so $|\phi(x_i)|> g(n)-f(n)$.
	Pick $\phi_Y \in \Phi_Y^\beta$ as follows:
	\begin{itemize}
		\item Primarily, minimize $|\phi(U_0x_1U_1x_2 \cdots U_{i-1}x_{i})|$;
		\item Secondarily, minimize $|\phi(U_i)|$;
		\item Tertiarily, minimize $|\phi(U_0x_1U_1x_2 \cdots U_{i-1})|$.
	\end{itemize}
	
	Roughly speaking, we have chosen $\phi_Y$ to move the image of $U_i$ as far left as possible in $Y$.
	But since $Y \not\in A_n^\alpha$, we want it further left!
	
	To suppress the details we no longer need, let $Y = Y_1\phi_Y(x_i) \phi_Y(U_i) \phi_Y(x_{i+1}) Y_2$, where $Y_1 = \phi_Y(U_0x_1U_1x_2 \cdots U_{i-1})$ and $Y_2 = \phi_Y(U_{i+1}x_{i+2} \cdots U_r)$. 
	
	Consider a word $Z \in \Gamma^n$ of the form $Y_1Z_1 \phi_Y(U_i) Z_2 \phi_Y(U_i) \phi_Y(x_{i+1}) Y_2$, where $Z_1$ is an initial string of $\phi_Y(x_{i})$ with $2f(n) \leq |Z_1| < g(n) - 2f(n)$ and $Z_2$ is a final string of $\phi_Y(x_{i})$. 
	(See Figure 1.)
	In a sense, the image of $x_i$ was too long, so we replace a leftward substring with a copy of the image of $U_i$.
	Let $C_Y$ be the set of all such $Z$ with $|Z_1|$ a multiple of $f(n)$.
	For every $Z \in C_Y$ we can see that $Z \in A_n^\alpha$, by defining $\psi \in \Phi_Z$ as follows:
	\[ \psi(y) = \left\{ \begin{array}{l l}
		Z_1 & \mbox{ if } y = x_{i}; \\
		Z_2 \phi_Y(U_i)\phi_Y(x_{i+1})& \mbox{ if } y = x_{i+1}; \\
		\phi_Y(y) & \mbox{ otherwise.}
	\end{array} \right. \]
	
	\begin{center}
		\begin{figure}[ht]
			\begin{tikzpicture}[scale=.43]
				\filldraw[fill=gray!20,draw=white](6,-1) rectangle (9,1);
				\draw(0,1)--(29,1)--(29,-1)--(0,-1)--(0,1);
				\draw(0,0)--(29,0);
				\draw(-1,.5) node {$Y=$};
				\draw(-1,-.5) node {$Z=$};
				\draw(4,1)--(4,-1.5);
				\draw(2,.5) node {$Y_1$};
				\draw(2,-.5) node {$Y_1$};
				\draw(14,1)--(14,-1);
				\draw(6,0)--(6,-1.5);
				\draw(9,0)--(9,-1.5);
				\draw(9,.5) node {$\phi_Y(x_i)$};
				\draw(5,-.5) node {$Z_1$};
				\draw(5,-1.5) node {$\psi(x_i)$};
				\draw(7.5,-.5) node {$\phi_Y(U_i)$};
				\draw(11.5,-.5) node {$Z_2$};
				\draw(17,1)--(17,-1);
				\draw(15.5,.5) node {$\phi_Y(U_i)$};
				\draw(15.5,-.5) node {$\phi_Y(U_i)$};
				\draw(23,1)--(23,-1.5);
				\draw(20,.5) node {$\phi_Y(x_{i+1})$};
				\draw(20,-.5) node {$\phi_Y(x_{i+1})$};
				\draw(14.5,-1.5) node {$\psi(x_{i+1})$};
				\draw(26,.5) node {$Y_2$};
				\draw(26,-.5) node {$Y_2$};
			\end{tikzpicture}
		
			\caption{Replacing a section of $\phi_Y(x_i)$ in $Y$ to create $Z$.}
		\end{figure}
	\end{center}
	
	
	\medskip
	Claim 1: $\displaystyle \liminf_{|Y| = n\rightarrow \infty} |C_Y| = \infty$.
	
	\medskip
	Since we want $2f(n) \leq |Z_1| < g(n) - 2f(n)$, and $g(n) - 2f(n) < |\phi_Y(x_i)| - |\phi_Y(U_i)|$, there are $g(n) - 4f(n)$ places to put the copy of $\phi_Y(U_i)$.
	To avoid any double-counting that might occur when some $Z$ and $Z'$ have their new copies of $\phi_Y(U_i)$ in overlapping locations, we further required that $f(n)$ divide $|Z_1|$. This produces the following lower bound: 
	\[ |C_Y| \geq \floor{\frac{g(n) - 4f(n)}{f(n)}} \rightarrow \infty. \]
	
	\medskip
	Claim 2: For distinct $Y, Y' \in A_n^\beta$, $C_Y \cap C_{Y'} = \emptyset$.
	
	\medskip
	To prove Claim 2, take $Y,Y' \in A_n^\beta$ with $Z \in C_Y \cap C_{Y'}$. Define $Y_1 = \phi_Y(U_0x_1U_1x_2 \cdots U_{i-1})$ and $Y_2 = \phi_Y(U_{i+1}x_{i+2} \cdots U_r)$ as before and $Y'_1 = \phi_{Y'}(U_0x_1U_1x_2 \cdots U_{i-1})$ and $Y'_2 = \phi_{Y'}(U_{i+1}x_{i+2} \cdots U_r)$. 
	Now for some $Z_1,Z'_1,Z_2,Z'_2$,
	\[ Y_1Z_1 \phi_Y(U_i) Z_2 \phi_Y(U_i) \phi_Y(x_{i+1}) Y_2 = Z = Y'_1Z'_1 \phi_{Y'}(U_i) Z'_2 \phi_{Y'}(U_i) \phi_{Y'}(x_{i+1}) Y'_2, \]
	with the following constraints:
	\begin{enumerate}
		\renewcommand{\theenumi}{\roman{enumi}}
		\item $|Y_1\phi_Y(U_i)| \leq |\phi_Y(UX_{i})| \leq f(n)$;
		\item $|Y'_1\phi_{Y'}(U_i)| \leq |\phi_{Y'}(UX_{i})| \leq f(n)$;
		\item $2f(n) \leq |Z_1| < g(n) - 2f(n)$;
		\item $2f(n) \leq |Z'_1| < g(n) - 2f(n)$;
		\item $|Z_1 \phi_Y(U_i) Z_2| = |\phi_{Y}(x_{i})| > g(n) - f(n)$;
		\item $|Z'_1 \phi_{Y'}(U_i) Z'_2| = |\phi_{Y'}(x_{i})| > g(n) - f(n)$.
	\end{enumerate}
	As a consequence:
	\begin{itemize}
		\item $|Y_1Z_1 \phi_Y(U_i)| < g(n) - f(n) < |Z'_1 \phi_{Y'}(U_i) Z'_2|$, by (i), (iii), and (vi);
		\item $|Y_1Z_1| \geq |Z_1| > 2f(n) > |Y'_1|$, by (iii) and (ii).
	\end{itemize}
	
	Therefore, the copy of $\phi_Y(U_i)$ added to $Z$ is properly within the noted occurrence of $Z'_1 \phi_{Y'}(U_i) Z'_2$ in $Z'$, which is in the place of $\phi_{Y'}(x_{i})$ in $Y'$.
	In particular, the added copy of $\phi_Y(U_i)$ in $Z$ interferes with neither $Y_1'$ nor the original copy of $\phi_{Y'}(U_i)$.
	Thus $Y_1'$ is an initial substring of $Y$ and $\phi_{Y'}(U_i) \phi_{Y'}(x_{i+1}) Y_2'$ is a final substring of $Y$.
	Likewise, $Y_1$ is an initial substring of $Y'$ and $\phi_Y(U_i) \phi_Y(x_{i+1}) Y_2$ is a final substring of $Y'$.
	By the selection process of $\phi_Y$ and $\phi_{Y'}$, we know that $Y_1 = Y'_1$ and $\phi_Y(U_i) \phi_Y(x_{i+1}) Y_2 = \phi_{Y'}(U_i) \phi_{Y'}(x_{i+1}) Y_2'$.
	Finally, since $f(n)$ divides $Z_1$ and $Z_1'$, we deduce that $Z_1 = Z_1'$. 
	Otherwise, the added copies of $\phi_Y(U_i)$ in $Z$ and of $\phi_{Y'}(U_i)$ in $Z'$ would not overlap, resulting in a contradiction to the selection of $\phi_Y$ and $\phi_{Y'}$.
	Therefore, $Y = Y'$, concluding the proof of Claim 2.
	
	
	Now $C_Y \subset A_n^\alpha$ for $Y \in A_n^\beta$. 
	Claim 1 and Claim 2 together imply that $|A_n^\beta| = o(|A_n^\alpha|)$.

\end{proof}

Observe that the choice of $\sqrt{n}$ in Lemma \ref{base} was arbitrary. 
The proof works for any function $f(n) = o(n)$ with $f(n) \rightarrow \infty$.
Therefore, where Lemma \ref{bulk} claims the existence of some $g(n) \rightarrow \infty$, the statement is in fact true for all $g(n) \rightarrow \infty$.

Let $\mathbb{I}_n(V,\Sigma)$ be the probability that a uniformly randomly selected length-$n$ $\Sigma$-word is an instance of $V$. 
That is, 
\[ \mathbb{I}_n(V,\Sigma) =\frac{|\{W \in \Sigma^n \mid \phi(V)=W \mbox{ for some homomorphism } \phi:{\rm L}(V)^* \rightarrow \Sigma^*\}|}{ |\Sigma|^{n}}. \]

\begin{fact} \label{EI}
	For any $V$ and $\Sigma$ and for $W_n \in \Sigma^n$ chosen uniformly at random, 
	\begin{eqnarray*}
		{n+1 \choose 2}\E(\delta(V,W_n)) & = & \sum_{m=1}^n (n+1-m)\E(\delta_{sur}(V,W_m))\\
		& = & \sum_{m=1}^n (n+1-m)\mathbb{I}_m(V,\Sigma).
	\end{eqnarray*}
\end{fact}

Denote $\mathbb{I}(V,\Sigma) = \lim_{n\rightarrow \infty}\mathbb{I}_n(V,\Sigma)$. When does this limit exist?

\begin{theorem} \label{nondoubledProb}
	For nondoubled $V$ and alphabet $\Sigma$, $\mathbb{I}(V,\Sigma)$ exists. Moreover, $\mathbb{I}(V,\Sigma) > 0$.
\end{theorem}

\begin{proof}
	If $|\Sigma|=1$, then $\mathbb{I}_n(V,\Sigma)=1$ for $n \geq |V|$.
	
	Assume $|\Sigma| = q \geq 2$.
	Let $V = TxU$ where $x$ is the right-most nonrecurring letter in $V$. Let $\Gamma = {\rm L}(V)$ be the alphabet of letters in $V$. 
	By Lemma \ref{bulk}, there is a nondecreasing function $g(n) = o(n)$ such that, for a randomly chosen $V$-instance $W \in \Sigma^n$, there is asymptotically almost surely a homomorphism $\phi: \Gamma^* \rightarrow \Sigma^*$ with $\phi(V) = W$ and $|\phi(x_r)| > n - g(n)$.
	
	Let $a_n$ be the number of $W \in \Sigma^n$ such that there exists $\phi: \Gamma^* \rightarrow \Sigma^*$ with $\phi(V) = W$ and $|\phi(x_r)| > n - g(n)$.
	Lemma \ref{bulk} tells us that $\frac{a_n}{q^n} \sim \mathbb{I}_n(V,\Sigma)$.
	Note that $\frac{a_n}{q^n}$ is bounded. 
	It suffices to show that $a_{n+1} \geq qa_n$ for sufficiently large $n$.
	Pick $n$ so that $g(n) < \frac{n}{3}$.

	For length-$n$ $V$-instance $W$ counted by $a_n$, let $\phi_W$ be a homomorphism that maximizing $|\phi_W(x_r)|$ and, of such, minimizes $|\phi_W(T)|$.
	For each $\phi_W$ and each $a \in \Sigma$, let $\phi_W^a$ be the function such that, if $\phi_W(x_r) = AB$ with $|A|  = \floor{|\phi_W(x_r)|/2}$, then $\phi_W^a(x) = AaB$; $\phi_W^a(y) = \phi_W(y)$ for each $y \in  \Gamma \setminus\{x\}$
	Roughly speaking, we are inserting $a$ into the middle of the image of $x$.
	
	Suppose we are double-counting, so $\phi_W^a(V) = \phi_Y^b(V)$.
	As 
	\[ |\phi_W(x_r)|/2 > (n - g(n))/2 > n/3 > g(n) \geq |\phi_Y(TU)| \]
	 and vice-versa, the inserted $a$ (resp., $b$) of one map does not appear in the image of $TU$ under the other map. 
	So $\phi_W(T)$ is an initial string and $\phi_W(U)$ a final string of $\phi_Y(V)$, and vice-versa.
	By the selection criteria of $\phi_W$ and $\phi_Y$, $|\phi_W(T)| = |\phi_Y(T)|$ and $|\phi_W(U)| = |\phi_Y(U)|$. 
	Therefore the location of the added $a$ in $\phi_W^a(V)$ and the added $b$ in $\phi_W^b(V)$ are the same.
	Hence, $a = b$ and $W = Y$.

	Moreover $\mathbb{I}(V,\Sigma) \geq q^{-||V||} > 0$.
\end{proof}

\begin{ex}
    Let $V = x_1x_2\cdots x_k$ have k distinct letters. 
    Since every word of length at least k is a $V$-instance, $\mathbb{I}(V,\Sigma)=1$ for every alphabet $\Sigma$.
    When even one letter in $V$ is repeated, finding $\mathbb{I}(V,\Sigma)$ becomes a nontrivial task.
\end{ex}

\begin{ex}
    Zimin's classification of unavoidable words is as follows \cite{Z-82,Z-84}: Every unavoidable word with $n$ distinct letters is encountered by $Z_n$, where $Z_0 = \varepsilon$ and $Z_{i+1} = Z_ix_{i+1}Z_i$ with $x_{i+1}$ a letter not occurring in $Z_i$.
    For example, $Z_2 = aba$ and $Z_3 = abacaba$.
    The authors can calculate $\mathbb{I}(Z_2,\Sigma)$ and $\mathbb{I}(Z_3,\Sigma)$ to arbitrary precision \cite{CR-16}.
    \begin{table}[h]
    	\caption{$\mathbb{I}(Z_2,\Sigma)$ and $\mathbb{I}(Z_3,\Sigma)$ computed to 7 decimal places.}
    	\medskip
    	\begin{tabular}{c | c | c | c | c | c | c | c | }
    		$|\Sigma|$ & 2 & 3 & 4 & 5 & 6 & 7 & $\cdots$ \\ \hline
        	$\mathbb{I}(Z_2,\Sigma)$ & 0.7322132 & 0.4430202 & 0.3122520 & 0.2399355 & 0.1944229 & 0.1632568 & $\cdots$\\ \hline
    		$\mathbb{I}(Z_3,\Sigma)$ & 0.1194437 & 0.0183514 & 0.0051925 & 0.0019974 & 0.0009253 & 0.0004857 & $\cdots$\\
    	\end{tabular}
	\end{table}
\end{ex}
    
\begin{cor}
	Let $V$ be a nondoubled word on any alphabet. 
	Fix an alphabet $\Sigma$, and let $W_n \in \Sigma^n$ be chosen uniformly at random. 
	Then 
	\[ \lim_{n\rightarrow \infty}\E(\delta(V,W_n)) = \mathbb{I}(V,\Sigma). \]
\end{cor}

\begin{proof}
	Let $\mathbb{I} = \mathbb{I}(V,\Sigma)$ and $\epsilon > 0$. 
	Pick $N = N_\epsilon$ sufficiently large so $|\mathbb{I} - \mathbb{I}_n(V,\Sigma)| < \frac{\epsilon}{2}$ when $n > N$. 
	Applying Fact \ref{EI} for $n > \max(N,4N/\epsilon)$,
	
	\begin{eqnarray*}
		|\mathbb{I} - \E(\delta(V,W_n))|& = & \left|\mathbb{I}\frac{1}{{n+1 \choose 2}}  \sum_{m=1}^n (n+1-m) - \frac{1}{{n+1 \choose 2}}  \sum_{m=1}^n (n+1-m)\mathbb{I}_m(V,\Sigma)\right|\\
		& \leq & \frac{1}{{n+1 \choose 2}}  \sum_{m=1}^n (n+1-m)|\mathbb{I} - \mathbb{I}_m(V,\Sigma)|\\
		& = & \frac{1}{{n+1 \choose 2}}  \left[\sum_{m=1}^N + \sum_{m = N+1}^n\right] (n+1-m)|\mathbb{I} - \mathbb{I}_m(V,\Sigma)|\\
		& < & \frac{1}{{n+1 \choose 2}}  \left[\sum_{m=1}^{\floor{\epsilon n / 4}} (n+1-m)1 + \sum_{m = N+1}^n (n+1-m)\frac{\epsilon}{2} \right]\\
		& < & \frac{1}{{n+1 \choose 2}}  \left[\frac{\epsilon n}{4}n + {n+1 \choose 2}\frac{\epsilon}{2} \right]\\
		& < & \epsilon.
	\end{eqnarray*}

\end{proof}

\section{Concentration}

For doubled $V$ and $|\Sigma|>1$, we established that the expectation of the density $\delta(V,W_n)$ converges to zero. 
In particular, we know the following.

\begin{prn} \label{expectation}
	Let $V$ be a doubled word, $\Sigma$ an alphabet with $q \geq 2$ letters, and $W_n \in \Sigma^n$ chosen uniformly at random. 
	Then \[ \E(\delta(V,W_n)) \sim \frac{1}{n}. \]
\end{prn}

\begin{proof}
	In the proof of Theorem \ref{main}, we showed that 
	\[ \E(\delta(V,W_n)) \leq \frac{\left(\sum_{j = k}^{\infty} {j+1 \choose k+1} q^{-j}\right)(n-|V|+1)}{{n+1 \choose 2}} = O(n^{-1}) . \]
	The lower bound follows from an observation made in the Background section: ``the event that $W_n[b|V|,(b+1)|V|]$ is an instance of $V$ has nonzero probability and is independent for distinct $b$.''
	Hence \[ \E(\delta(V,W_n)) \geq \frac{ 1}{{n + 1 \choose 2}}\floor{\frac{n}{|V|}}\mathbb{I}_{|V|}(V,\Sigma) = \Omega(n^{-1}). \]
\end{proof}


To bound variance and other higher order moments, we observe the following upper bound on $q^n \mathbb{I}_n(V,\Sigma)$. 
Hencefore, if ${x \choose y}$ is used with nonintegral $x$, we mean \[ {x \choose y} = \frac{\prod_{i = 0}^{y-1} (x-i)}{y!}. \]

\begin{lem} \label{instanceCount}
	Let $V$ be a doubled word with exactly $k$ letters and $\Sigma$ an alphabet with $q \geq 2$ letters. 
	Moreover, let ${\rm L}(V) = \{x_1, \ldots, x_k\}$ with $r_i$ be the multiplicity of $x_i$ in $V$ for each $i \in [k]$,  $d = \gcd_{i \in [k]}(r_i)$, and $r = \min_{i \in [k]}(r_i)$.
Then,

\[ \mathbb{I}_{n}(V,\Sigma) \leq {n/d+k+1 \choose k+1} q^{n(1-r)/r}. \]
\end{lem}

\begin{proof}
	Let $a_n(\overline{r})$ be the number of $k$-tuples $\overline{a} = (a_1, \cdots, a_k) \in (\Z^+)^k$ so that $\sum_{i=1}^{k} a_ir_i = n$. 
	Then $a_{n}(\overline{r}) \leq {n/d+k +1 \choose k +1 }$. 
	Indeed, if $d \! \not | \; n$, then $a_{n}(\overline{r}) = 0$. 
	Otherwise, for each $\overline{a}$ counted by $a_{n}(\overline{r})$, there is a unique corresponding $\overline{b}\in (\Z^+)^k$ such that $1 \leq b_1 < b_2 < \cdots < b_k = n/d$ and $b_j = \frac{1}{d} \sum_{i = 1}^{j} a_ir_i$. 
	The number of strictly increasing $k$-tuples of positive integers with largest value $n/d$ is ${n/d+k+1 \choose k+1}$.
	Let $W_n \in \Sigma^n$ chosen uniformly at random. 
	Note that $q^n \mathbb{I}_n(V,\Sigma)$ is the number of instances of $V$ in $\Sigma^n$. 
Thus,

\[ q^{n}\mathbb{I}_{n}(V,\Sigma) \leq  \E(\hom(V,W_{n})) < {n/d+k+1 \choose k+1} q^{n/r}. \]
\end{proof}


We obtain nontrivial concentration around the mean using covariance and the fact that most ``short'' substrings in a word do not overlap. 

\begin{theorem} \label{var}
	Let $V$ be a doubled word with $k$ distinct letters, $\Sigma$ an alphabet with $q \geq 2$ letters, and $W_n \in \Sigma^n$ chosen uniformly at random. 
	\[ {\rm Var}(\delta(V,W_n)) =O\left(\E(\delta(V,W_n))^2 \frac{(\log n)^3}{n}\right). \]
\end{theorem}

\begin{proof}
	Let $X_n = {n+1 \choose 2}\delta(V,W_n)$ be the random variable counting the number of substrings of $W_n$ that are $V$-instances. 
	For fixed $n$, let $X_{a,b}$ be the indicator variable for the event that $W_n[a,b]$ is a $V$-instance, so $X_n = \sum_{a=0}^{n-1} \sum_{b=a+1}^{n} X_{a,b}$. 
	Let $(a,b) \sim (c,d)$ denote that $[a,b]$ and $[c,d]$ overlap.
	Note that 
	\begin{eqnarray*}
		{\rm Cov}(X_{a,b},X_{c,d})& \leq & \E(X_{a,b}X_{c,d})\\
		& \leq & \min(\E(X_{a,b}),\E(X_{c,d})) \\
		& = & \min(\mathbb{I}_{(b-a)}(V,\Sigma),\mathbb{I}_{(b-a)}(V,\Sigma))\\
		& \leq & {i/d+k+1 \choose k+1}q^{i(1-r)/r},
	\end{eqnarray*}
	for $i \in \{b-a,d-c\}$. 
	For $i < n/3$, the number of intervals in $W_n$ of length at most $i$ that overlap a fixed interval of length $i$ is less than ${3i \choose 2}$.
	Define the following function on $n$, which acts as a threshold for ``short'' substrings of a random length-$n$ word: 
	\[ s(n) = -2 \log_q(n^{-(k+5)}) = t \log n, \]
	 where $t = \frac{2(k+5)}{\log(q)} > 0$.
	For sufficiently large $n$,

    \begin{eqnarray*}
    	{\rm Var}(X_n) & = & \sum_{\substack{0 \leq a < b \leq n \\ 0 \leq c < d \leq n}} {\rm Cov}(X_{a,b},X_{c,d})\\
    	& \leq & \sum_{(a,b) \sim (c,d)} \min(\mathbb{I}_{(b-a)}(V,\Sigma),\mathbb{I}_{(b-a)}(V,\Sigma))  \\
    	& = & \left[ \sum_{\substack{(a,b) \sim (c,d) \\ b-a,d-c \leq s(n)}} + \sum_{\substack{(a,b) \sim (c,d) \\ else}} \right] \min(\mathbb{I}_{(b-a)}(V,\Sigma),\mathbb{I}_{(b-a)}(V,\Sigma)) \\
    	& < & 2\sum_{i = 1}^{\floor{s(n)}}(n+1-i){3i \choose 2}\cdot 1 \\
    	& &  + \sum_{i = \ceil{s(n)}}^{n} (n+1-i) {n+1 \choose 2}\cdot {i/d+k+1 \choose k+1}q^{i(1-r)/r} \\
    	& < &2s(n)n(3s(n))^2 + nnn^2n^{k+1}q^{s(n)(1-r)/r}\\
    	& = &18(t \log n)^3n + n^{5+k} q^{\log_q\left(n^{-(k+5)}\right)}\\\
    	& = & O(n(\log n)^3).\\
    \end{eqnarray*}
    
    Since $\E(\delta(V,W_n)) = \Omega(n^{-1})$ by Corollary \ref{expectation},
    
    \begin{eqnarray*}
    	{\rm Var}(\delta(V,W_n)) & = & {\rm Var}\left(\frac{X_n}{{n+1 \choose 2}}\right)\\
    	& = & \frac{{\rm Var}(X_n)}{{n+1 \choose 2}^2}\\
    	& = & O\left(\frac{(\log n)^3}{n^3}\right)\\
    	& = &O\left(\E(\delta(V,W_n))^2 \frac{(\log n)^3}{n}\right).
    \end{eqnarray*}

\end{proof}

\begin{lem} \label{prob}
	Let $V$ be a word with $k$ distinct letters, each occurring at least $r \in \Z^+$ times. 
	Let $\Sigma$ be a $q$-letter alphabet and $W_n \in \Sigma^n$ chosen uniformly at random.
	Recall that ${n+1 \choose 2}\delta(V,W_n)$ is the number substrings of $W_n$ that are $V$-instances.
	Then for any nondecreasing function $f(n) > 0$, 
	\[ \mathbb{P}\left({n+1 \choose 2}\delta(V,W_n) > n\cdot f(n) \right) < n^{k+3}q^{f(n)(1-r)/r}. \]
\end{lem}

\begin{proof}
	Lemma~\ref{instanceCount} gives a bound on the probability that randomly chosen $W_n \in \Sigma^n$ is a $V$-instance: 
	\[ \mathbb{P}(\delta_{sur}(V,W_n) = 1) = \mathbb{I}_{n}(V,\Sigma) \leq {n/d+k+1 \choose k+1}q^{n(1-r)/r}.\] 

    Since $\delta_{sur}(V,W) \in \{0,1\}$,
	\begin{eqnarray*}
	    \sum_{m=1}^{\floor{f(n)}} \sum_{\ell=0}^{n-m}\delta_{sur}(V,W_n[\ell,\ell+m]) & < & n \cdot f(n).
	\end{eqnarray*}
	Therefore,
	\begin{eqnarray*}
		\mathbb{P}\left({n+1 \choose 2}\delta(V,W_n) > n\cdot f(n) \right) & = & \mathbb{P}\left(\sum_{m=1}^n \sum_{\ell=0}^{n-m}\delta_{sur}(V,W_n[\ell,\ell+m]) > n\cdot f(n) \right) \\
		& < & \mathbb{P}\left(\sum_{m=\ceil{f(n)}}^n \sum_{\ell=0}^{n-m}\delta_{sur}(V,W_n[\ell,\ell+m])> 0 \right) \\
		& < & \sum_{m=\ceil{f(n)}}^n \sum_{\ell=0}^{n-m}\mathbb{P}\left(\delta_{sur}(V,W_n[\ell,\ell+m])> 0 \right) \\
		& = & \sum_{m=\ceil{f(n)}}^n (n-m+1) \mathbb{P}\left(\delta_{sur}(V,W_m) = 1 \right) \\
		& \leq & \sum_{m=\ceil{f(n)}}^n (n-m+1){m/d+k+1 \choose k+1} q^{m(1-r)/r}\\
		& < & n(n-m+1) {n/d+k+1\choose k+1} q^{f(n)(1-r)/r}\\
		& < & n^{k+3}q^{f(n)(1-r)/r}.
	\end{eqnarray*}
\end{proof}

\begin{theorem} \label{moments}
	Let $V$ be a doubled word, $\Sigma$ an alphabet with $q \geq 2$ letters, and $W_n \in \Sigma^n$ chosen uniformly at random. 
	Then the $p^{th}$ raw moment and the $p^{th}$ central moment of $\delta(V,W_n)$ are both $O\left(\left(\log(n)/n\right)^p\right)$.
\end{theorem}

\begin{proof}
	Let us use Lemma \ref{prob} to first bound the $p$-th raw moments for $\delta(V,W_n)$, assuming $r\geq 2$.
	To minimize our bound, 
	generalize the threshold function from Theorem~\ref{var}:
	\[ s_p(n) = \frac{r}{1-r}\log_q(n^{-(k+5+p)}) = t_p \log n, \] 
	where $t_p = \frac{r(k+5+p)}{(r-1)\log(q)} > 0$.
	\begin{eqnarray*}
		\E(\delta(V,W_n)^p) & = & \sum_{i = 0}^{{n+1 \choose 2}} \mathbb{P}\left(\delta(V,W_n) = \frac{i}{{n+1 \choose 2}}\right)\left(\frac{i}{{n+1 \choose 2}}\right)^p\\
		& < & \sum_{i = 0}^{\floor{n\cdot s_p(n)}} \mathbb{P}\left(\delta(V,W_n) = \frac{i}{{n+1 \choose 2}}\right)\left(\frac{i}{{n+1 \choose 2}}\right)^p \\&&+ \sum_{i = \ceil{n\cdot s_p(n)}}^{{n+1 \choose 2}} n^{k+3}q^{s_p(n)(1-r)/r}\left(\frac{i}{{n+1 \choose 2}}\right)^p\\
		& < & \left(\frac{n\cdot s_p(n)}{{n+1 \choose 2}}\right)^p +  n^{k+5 }q^{s_p(n)(1-r)/r}\\
		& = & \left(\frac{nt_p \log n}{{n+1 \choose 2}}\right)^p +  n^{k+5 }q^{\log_q\left(n^{-(k+5+p)}\right)}\\
		&= &O_p\left(\left(\frac{\log n}{n}\right)^p\right).
	\end{eqnarray*}
	
	Setting $p=1$, there exists some $c>2$ such that $\E_n= \E(\delta(V,W_n)) < (c\log n)/n$. We use this upper bound on the expectation (1st raw moment) to bound the central moments. 
	
	\begin{eqnarray*}
		\E(\left|\delta(V,W_n) - \E_n\right|^p) & = &  \sum_{i = 0}^{{n+1 \choose 2}} \mathbb{P}\left(\delta(V,W_n) = \frac{i}{{n+1 \choose 2}}\right)\left|\frac{i}{{n+1 \choose 2}} - \E_n \right|^p\\
		& \leq &  \sum_{i = 0}^{\floor{n\cdot s_p(n)}}  \mathbb{P}\left(\delta(V,W_n) = \frac{i}{{n+1 \choose 2}}\right)\left(\frac{c\log n}{n}\right)^p \\
			& & + \sum_{i = \ceil{ns_p(n)}}^{{n+1 \choose 2}} \mathbb{P}\left(\delta(V,W_n) = \frac{i}{{n+1 \choose 2}}\right)\left(1\right)^p\\
		& < &\left(\frac{c\log n}{n}\right)^p + n^{k+5}q^{s_p(n)(1-r)/r}\\
		&= &O_p\left(\left(\frac{\log n}{n}\right)^p\right).
	\end{eqnarray*}
\end{proof}

\begin{qun}
For nondoubled word $V$, to what extent is the density of $V$ in random words concentrated about its mean?
\end{qun}

\bibliography{Rorabaugh-refs}{}

\providecommand{\bysame}{\leavevmode\hbox to3em{\hrulefill}\thinspace}
\providecommand{\MR}{\relax\ifhmode\unskip\space\fi MR }
\providecommand{\MRhref}[2]{%
  \href{http://www.ams.org/mathscinet-getitem?mr=#1}{#2}
}
\providecommand{\href}[2]{#2}
\begin{thebibliography}{10}

\bibitem{BEM-79}
D.~R. Bean, A.~Ehrenfeucht, and G.~F. McNulty, \emph{{Avoidable Patterns in
  Strings of Symbols}}, Pac. J. of Math. \textbf{85:2} (1979), 261--294.

\bibitem{BG-07}
J.~P. Bell and T.~L. Goh, \emph{Exponential lower bounds for the number of
  words of uniform length avoiding a pattern}, Information and Computation
  \textbf{205} (2007), 1295--1306.

\bibitem{BW-12}
F.~Blanchet-Sadri and B.~Woodhouse, \emph{Strict bounds for pattern avoidance},
  Theoretical Computer Science \textbf{506} (2013).

\bibitem{CR-14}
J.~Cooper and D.~Rorabaugh, \emph{{Bounds on Zimin Word Avoidance}}, Congressus
  Numerantium \textbf{222} (2014), 87--95.

\bibitem{CR-16}
\bysame, \emph{{Asymptotic Density of Zimin Words}}, Discrete Mathematics \&
  Theoretical Computer Science \textbf{18:3\#3} (2016).

\bibitem{C-05}
J.~D. Currie, \emph{Pattern avoidance: themes and variations}, Theoretical
  Computer Science \textbf{339} (2005).

\bibitem{L-02}
M.~Lothaire, \emph{{Algebraic Combinatorics on Words}}, Cambridge University
  Press, Cambridge, 2002.

\bibitem{L-12}
L.~Lov{\'a}sz, \emph{{Large Networks and Graph Limits}}, American Mathematical
  Society, Providence, 2012.

\bibitem{T-14}
J.~Tao, \emph{Pattern occurrence statistics and applications to the {R}amsey
  theory of unavoidable patterns}, arXiv:1406.0450.

\bibitem{T-06}
A.~Thue, \emph{{\"U}ber unendliche {Z}eichenreihen}, Norske Vid. Skrifter I
  Mat.-Nat. Kl., vol.~7, Kristiania, 1906.

\bibitem{Z-82}
A.~I. Zimin, \emph{Blokirujushhie mnozhestva termov}, Mat. Sb. \textbf{119}
  (1982), 363--375.

\bibitem{Z-84}
\bysame, \emph{Blocking sets of terms}, Math. USSR-Sb. \textbf{47} (1984),
  353--364.

\end{thebibliography}
\bibliographystyle{amsplain}

\label{lastpage}

\end{document}